\documentclass[12pt, twoside, leqno]{article}

\usepackage{amsmath,amsthm}
\usepackage{amssymb}

\usepackage{enumitem}

\usepackage{graphicx}

\usepackage[T1]{fontenc}

\newtheorem{theorem}{Theorem}[section]
\newtheorem{corollary}[theorem]{Corollary}
\newtheorem{lemma}[theorem]{Lemma}

\theoremstyle{definition}

\newtheorem{example}[theorem]{Example}

\numberwithin{equation}{section}

\def\R{{\mathbb R}}

\def\E{{\mathbb E}}

\begin{document}


\baselineskip=17pt


\title{Real Zeros of Random Sums with I.I.D. Coefficients}

\author{Aaron M. Yeager\\
Department of Mathematics\\
College of Coastal Georgia\\
One College Drive\\
Brunswick, GA 31520\\
E-mail: aaronyeager235@gmail.com}

\date{}

\maketitle


\renewcommand{\thefootnote}{}

\footnote{2010 \emph{Mathematics Subject Classification}: Primary 30C15, 30B20; Secondary 26C10, 60B99.}

\footnote{\emph{Key words and phrases}: Random Polynomials, Fourier Analysis, Orthogonal Polynomials, Bergman Polynomials.}

\renewcommand{\thefootnote}{\arabic{footnote}}
\setcounter{footnote}{0}


\begin{abstract}
Let $\{f_k\}$ be a sequence of entire functions that are real valued on the real-line.  We study the expected number of real zeros of random sums of the form
$P_n(z)=\sum_{k=0}^n\eta_k f_k(z)$, where $\{\eta_k\}$ are real valued i.i.d.~random variables.  We establish a formula for the density function $\rho_n$ for the expected number of real zeros of $P_n$.  As a corollary, taking the random variables $\{\eta_k\}$ to be i.i.d.~standard Gaussian, appealing to Fourier inversion we recover the representation for the density function previously given by Vanderbei through means of a different proof.  Placing the restrictions on the common characteristic function $\phi$ of $\{\eta_k\}$ that $|\phi(s)|\leq (1+as^2)^{-q}$, with $a>0$ and $q\geq 1$, as well as that $\phi$ is three times differentiable with each the second and third derivatives being uniformly bounded, we achieve an upper bound on the density function $\rho_n$ with explicit constants that depend only on the restrictions on $\phi$.  As an application we considered the limiting value of $\rho_n$ when the spanning functions $f_k(z)=p_k(z)$, $k=0,1,\dots, n$, where $\{p_k\}$ are Bergman polynomials on the unit disk.
\end{abstract}

\section{Introduction}

The systematic study of the expected number of real zeros of polynomials
\begin{equation*}
P_n(z)=\eta_nz^n+\eta_{n-1}z^{n-1}+\cdots +\eta_1 z + \eta_0
\end{equation*}
with random coefficients $\{\eta_j\}$, called \emph{random algebraic polynomials} (or \emph{Kac polynomials}), dates back to the early 1930's.  For early results in this area we refer the reader to the works \cite{BP}, \cite{LO1}, \cite{LO2}, \cite{LO3}, \cite{LO4}, \cite{LO5}, as well as the books by Bharucha-Reid and Sambandham \cite{BR} and Farahmand \cite{Fa}.

Let $\E$ denote the mathematical expectation and $N_n(S)$ denote the number of zeros of $P_n$ in a set $S$.  In 1943, for a measurable $\Omega \subset \R$, Kac \cite{K1} produced an integral equation for  $\E[N_n(\Omega)]$ for the random algebraic polynomial $ P_n $ when the random variables $\{\eta_j\}$ are i.i.d.~standard  Gaussian.
Independently while studying random noise in 1945, Rice \cite{R} derived a similar formula for $\E[N_n(\R)]$ in the Gaussian setting.
After Kac established the integral equation for $\E[N_n(\Omega)]$, he proved the asymptotic
\begin{equation}
\label{real-expected-Kac}
\E[N_n(\R)] = \frac{2+o(1)}\pi\log n \quad \text{as} \quad n\to\infty.
\end{equation}
The error term in the above asymptotic was further  sharpened by Hammersley \cite{HM}, Edelman and Kostlan \cite{EK}, and finally Wilkins \cite{WL}.

Kac conjectured that a similar asymptotic as \eqref{real-expected-Kac} should hold when the random variables are i.i.d.~ uniform on $[-1,1]$ following his same proof.  Realizing that the same proof would not go through, in \cite{K2} Kac was able to produce the asymptotic \eqref{real-expected-Kac} in the uniform random variable case. For other results concerning non-Gaussian random algebraic polynomials see \cite{EO}, \cite{IM1}, \cite{IM2}, \cite{IZ}, \cite{LS1}, \cite{LS2}, \cite{PritYgr15}, and \cite{NNV}.

Due to the work of Kac and Rice, formulas for the density function that give the expected number of real zeros of the random polynomial $P_n$ of the form
\begin{equation}\label{Kac-Rice}
\rho_n(x)=\int_{\R} |\eta| D_n(0,\eta;x)\ d\eta,
\end{equation}
where $D_n(\xi,\eta;x)$ is the joint density distribution of $P_n(x)$ and $P_n^{\prime}(x)$, with
\begin{equation*}
\E[N_n(\Omega)]=\int_{\Omega}\rho_n(x) \ dx
\end{equation*}
are called \emph{Kac-Rice formulas}.  We note that such formulas are also referred to as the \emph{intensity function} or the first \emph{correlation function}.

When dealing with the expected number of real zeros of 
$$P_n(z)=\sum_{j=0}^n\eta_jf_j(z),$$
where $\{f_j\}$ are any thing other than the monomials, and $\{\eta_j\}$ are non-Gaussian random variables, the Kac-Rice formula \eqref{Kac-Rice} still holds.  However the evaluation of this formula is very difficult.  In fact little is known about a workable shape of the intensity function in this non-Gaussian setting.

Instead of altering the spanning functions to not be the monomials, many authors have remained with the monomial basis and introduced weights that can help with asymptotics of the intensity function.  In this case the random sums take the shape
\begin{equation}\label{wrs}
G_n(z)=\sum_{j=0}^n\eta_j c_j z^j,
\end{equation}
where $\{\eta_j\}$ are i.i.d.~random variables, and $\{c_j\}$ are deterministic weights.  For results concerning the weighted random polynomial $G_n$ we direct the reader to the works  \cite{EK}, \cite{TV}, and \cite{DNV2}.

Appealing to Fourier transforms of distribution functions, Bleher and Di \cite{BD} gave a universality result for the expected number of real zeros of $G_n$ defined in \eqref{wrs} when coefficients $\{c_i\}$ are elliptical weights, that is weights of the form $\sqrt{n\choose i}$, and $\{\eta_i\}$ are i.i.d.~random variables with mean zero and variance one.  To achieve their result, Bleher and Di assume that the common characteristic function for the i.i.d.~random variables
\begin{equation*}
\phi(s)=\int_{\R}r(t)e^{its}\ dt
\end{equation*}
satisfies
$|\phi(s)|\leq (1+as^2)^{-q}$ with $a,q>0$,
and
$\sup_{s\in \R}|\frac{d^j}{ds^j}\phi(s)|\leq c_j$, for $j=2,3$,
where $c_2,c_3>0$ are constants. Under these assumptions, for $x\neq 0$ they show that
\begin{equation}\label{BLD}
\lim_{n\rightarrow \infty}\frac{\rho_n(x)}{\sqrt{n}}=\frac{1}{\pi(1+x^2)},
\end{equation}
where $\rho_n$ is the intensity function for the random sum $G_n$.  Under further assumptions on the shape of characteristic function $\phi(s)$ and that the derivatives up to the the sixth order are bounded, they show that \eqref{BLD} also holds for $x=0$.  In light of the work by Edelmon and Kostlan \cite{EK}, the result \eqref{BLD} matches up with the case when the random variables of the weighted random sum are i.i.d.~ standard  Gaussian.  The technique that Bleher and Di use also allows them to extend their result to higher order correlation functions and to non-Gaussian multivariate weighed random polynomials.

Applying techniques given by Bleher and Di in \cite{BD}, we achieve a workable representation for the density function of the expected number of zeros of a random sum spanned by entire functions that are real valued on the real-line.  To specify these results,  let $\{\eta_k\}$ be real valued i.i.d.~ random variables such that
\begin{equation}\label{cs}
\E[\eta_k]=0,\ \ \ \E[\eta^2_k]=1,\ \ \ k=0,1,\dots,n.
\end{equation}
Consider the random linear combination
\begin{equation}\label{Pn}
P_n(z)=\eta_nf_n(z)+\eta_{n-1}f_{n-1}(z)+\cdots+\eta_1f_1(z)+\eta_0f_0(z),
\end{equation}
with $f_n,f_{n-1},\dots, f_0$ being entire functions that are real valued on the real-line.  Let $\rho_n(x)$ be density function for the expected
number of real zeros of $P_n(x)$. Define
\begin{equation}\label{Kn}
\mathcal K_n(x):=\frac{\sqrt{K_n^{(1,1)}(x,x)K_n(x,x)-\left(K_n^{(0,1)}(x,x)\right)^2  }}{ K_n(x,x)} ,
\end{equation}
where
\begin{align}\nonumber
K_n(x,x)=\sum_{j=0}^n&f_j(x)^2,\  \ \ K_n^{(0,1)}(x,x)=\sum_{j=0}^nf_j(x)f_j^{\ \prime}(x), \\
\label{kernels}
& K_n^{(1,1)}(x,x)=\sum_{j=0}^nf_j^{\ \prime}(x)^2.
\end{align}

\begin{theorem}\label{denlammu}
The density function $\rho_n(x)$ of the real zero distribution of the random sum $P_n(x)$ given by \eqref{Pn} can be written as
\begin{equation}\label{rhon}
\rho_n(x)=\mathcal K_n(x)\int_{\R}|\eta|\ \widehat{D}_n(0,\eta;x)\ d\eta,
\end{equation}
where $\widehat{D}_n(\xi,\eta;x)$ is the joint distribution density of the random variables
$$g_n(x)=\sum_{k=0}^n\mu_k(x)\eta_k,\ \ \ h_n(x)=\sum_{k=0}^n\lambda_k(x)\eta_k,$$
with
$$\mu_k(x)=\frac{f_k(x)}{\left( K_n(x,x) \right)^{1/2}},$$
and
$$ \lambda_k(x)=\frac{K_n(x,x)f_k^{\
\prime}(x)-K_n^{(0,1)}(x,x)f_k(x)}{\left[K_n(x,x)\left(K_n^{(1,1)}(x,x)K_n(x,x)-K_n^{(0,1)}(x,x)^2  \right)  \right]^{1/2}}.$$
Furthermore, $\{\mu_k\}$ and $\{\lambda_k\}$ satisfy
\begin{equation}\label{proplammu}
\sum_{k=0}^n\mu_k(x)^2=\sum_{k=0}^n\lambda_k(x)^2=1, \ \ \ \text{and} \ \ \ \sum_{k=0}^n\lambda_k(x)\mu_k(x)=0.
\end{equation}
\end{theorem}

\begin{corollary}\label{intGauss}
Let the common characteristic function of the i.i.d.~random variables $\{\eta_k\}$ be $\phi(s)=\exp(-as^2)$, where $a\in (0,\infty)$ is any fixed number.  Then \eqref{rhon} satisfies
\begin{equation}\label{Gaussian}
\rho_n(x)=\frac{1}{\pi}\mathcal K_n(x).
\end{equation}
\end{corollary}
In particular, when $a=1/2$, that is when the random variables $\{\eta_k\}$ are i.i.d.~ standard Gaussian, the above theorem recovers the result proven by Vanderbei (Theorem 1.2 \cite{CZRS}).  The proof of Corollary \ref{intGauss} differs from Vanderbei's proof in that it uses the representation \eqref{rhon} and Fourier Transforms instead of relying on the argument principle.  Furthermore, following our approach allows \eqref{Gaussian} to hold for i.i.d.~ scaled mean zero Gaussian random variables.

For the next result we will need some assumptions on the common characteristic function $\phi(s)$ of the i.i.d.~ random variables $\{\eta_k\}$.
Assume that $\phi(s)$ satisfies the following: for fixed $a>0$ and $q\geq 1$,
\begin{equation}\label{hyphi1}
|\phi(s)|\leq \frac{1}{(1+as^2)^q}, \ \ s\in \R,
\end{equation}
and that $\phi(s)$ is a three times differentiable function with there existing constants $C_2, C_3>0$ such that
\begin{equation}\label{hyphi2}
\sup_{-\infty<s<\infty}\left| \frac{d^j\phi(s)}{ds^j} \right|\leq C_j, \ \ \ j=2,3.
\end{equation}

\begin{theorem}\label{mainthm}
Suppose that the characteristic function for the collection of i.i.d.~random variables $\{\eta_i\}$ possess conditions \eqref{hyphi1} and \eqref{hyphi2}.  Then the density function
$\rho_n(x)$ of the real zero distribution of the random sum $P_n(x)$ to satisfies
\begin{equation}
\rho_n(x)\leq \mathcal K_n(x)\frac{1}{aq}\left[k_1+C_3k_2+\frac{1}{\sqrt{aq}} C_2(k_3+C_2k_4) \right].
\end{equation}
Here $k_1,k_2,k_3,k_4$ are constants that depends only on the conditions \eqref{hyphi1} and \eqref{hyphi2},
with $k_1\approx 0.36$, $k_2\approx 0.27$, $k_3\approx 0.21$, and $k_4\approx 1.18$.
\end{theorem}

We note that the constants $k_1$, $k_2$, $k_3$, and $k_4$ in Theorem \ref{mainthm} are given explicitly in the proof.

\begin{example}
Consider a Laplace distribution with characteristic function of the form
\begin{equation*}
\phi(s)=\frac{1}{1+s^2/2}.
\end{equation*}
Note that such a function is always greater than or equal to the Gaussian characteristic function $\exp(-s^2/2)$.  In this example we have $C_2=1$ and $C_3\approx 1.65$, so that along with Theorem \ref{mainthm} we see that
\begin{equation*}
\rho_n(x)\leq \mathcal K_n(x)\times 5.56174\dots =\frac{1}{\pi}\mathcal K_n(x)\times 17.4727\dots.
\end{equation*}
Thus in light of Corollary \ref{Gaussian}, it follows that $\rho_n(x)$ is at most $17.4727\dots$ larger than the case when the random variables are from the Gaussian distribution with characteristic function $\exp(-as^2)$, with $a>0$.
\end{example}

Asymptotics for the density function $\rho_n(x)$ in the case when the random variables $\{\eta_k\}$ are i.i.d.~standard Gaussian has been well studied when the spanning functions are trigonometric functions \cite{CZRS}, polynomials orthogonal on the real line (\cite{D},\cite{DB}, \cite{BY}, \cite{LPX}, \cite{LPX2}, \cite{PRIT}, \cite{AY2}), and polynomials orthogonal on the unit circle (\cite{AY2}, \cite{HY}, \cite{YY}).  As an application we consider the case $P_n(x)=\sum_{k=0}^n\eta_k f_k(x)$ with $f_k(x)=p_k(x)$, where $p_k(z)=\sqrt{(k+1)\pi}z^k$ are Bergman polynomials on the unit disk, i.e.~polynomials orthogonal with respect to area measure over the unit disk.
\begin{theorem}\label{Berg1}
Let $f_k(x)=p_k(x)=\sqrt{(k+1)/\pi}z^k$, $k=0,\dots, n$, where $\{p_k\}$ are Bergman polynomials on the unit disk.  Then the function $\mathcal K_n(x)$ defined at \eqref{Kn} possess
\begin{equation*}
\lim_{n\rightarrow \infty}\mathcal K_n(x)=\begin{cases}
      \displaystyle\frac{\sqrt{2}}{1-x^2} & |x|<1, \\[1.5ex]
      \displaystyle\frac{1}{x^2-1} & |x|>1.
   \end{cases}
\end{equation*}
Furthermore, the above convergence holds locally uniformly on the respective domains and on the boundary we have
$$\mathcal K_n(\pm 1)=\frac{1}{3}\sqrt{\frac{n(n+3)}{2} }.$$
\end{theorem}

The above theorem in connection with Theorem \ref{intGauss} allows one to know the limiting value of the intensity function for the random sum $P_n(z)$ when the random variables are scaled Gaussian, and also gives an upper bound for the limiting value of the intensity function in the non-Gaussian setting of Theorem \ref{mainthm}.

\section{The Proofs}
\subsection{Proof of Theorem \ref{denlammu}}

Since the coefficients $\eta_0, \eta_1,\dots, \eta_n$ satisfy \eqref{cs}, we have
\begin{align}
\nonumber
\E[P_n(x)]&=\E[P_n^{\ \prime}(x)]=0,\\
\nonumber
\E[P_n(x)^2]&=\sum_{k=0}^nf_k(x)^2=K_n(x,x),\\
\nonumber
\E[P_n(x)P_n^{\ \prime}(x)]&=\sum_{k=0}^nf_k(x)f_k^{\ \prime}(x)=K_n^{(0,1)}(x,x),\\
\label{Kn01}
\E[P_n^{\ \prime}(x)^2]&=\sum_{k=0}^nf_k^{\ \prime}(x)^2=K_n^{(1,1)}(x,x).
\end{align}
Following the method of Bleher and Di in \cite{BD}, we now rescale $P_n(x)$ and $P_n^{\ \prime}(x)$ as follows:
\begin{align}\nonumber
g_n(x)&:=\frac{P_n(x)}{\sqrt{K_n(x,x)}}=\sum_{k=0}^n\mu_k(x)\eta_k\\
\label{gn}
 \widetilde{g}_n(x)&:=\frac{P_n^{\
\prime}(x)}{\sqrt{K_n^{(1,1)}(x,x)}}=\sum_{k=0}^n\nu_k(x)\eta_k,
\end{align}
where
\begin{equation}\label{munu}
\mu_k(x)=\frac{f_k(x)}{\sqrt{K_n(x,x)}}\ \ \ \ \ \text{and}\ \ \ \ \ \nu_k(x)=\frac{f_k^{\ \prime}(x)}{\sqrt{K_n^{(1,1)}(x,x)}},\ \ \ \ k=0,1,\dots, n.
\end{equation}
Let $\widetilde{D}_n(\xi,\eta;x)$ be the joint distribution density of $g_n(x)$ and $\widetilde{g}_n(x)$.  By a change of variables we have
$$D_n(\xi,\eta;x)=\frac{1}{\sqrt{K_n(x,x)K_n^{(1,1)}(x,x)}}\ \widetilde{D}_n\left(\frac{\xi}{\sqrt{K_n(x,x)}},\frac{\eta}{\sqrt{K_n^{(1,1)}(x,x)}};x\right),$$
so that the Kac-Rice equation \eqref{Kac-Rice} is now
\begin{equation}\label{density2}
\rho_n(x)=\sqrt{\frac{K_n^{(1,1)}(x,x)}{K_n(x,x)}}\int_{\R}|\eta|\ \widetilde{D}_n(0,\eta;x)\ d\eta.
\end{equation}
Observe that \eqref{munu} and \eqref{Kn01} give
\begin{equation}\label{munu21}
\sum_{k=0}^n\mu_k(x)^2=\sum_{k=0}^n\nu_k(x)^2=1,\ \ \ \text{and} \ \ \ \sum_{k=0}^n\mu_k(x)\nu_k(x)=\frac{K_n^{(0,1)}(x,x)}{\sqrt{K_n(x,x)K_n^{(1,1)}(x,x)}}.
\end{equation}
We now change the joint distribution density $\widetilde{D}_n(\xi,\eta;x)$ for $g_n(x)$ and $\widetilde{g}_n(x)$ to that of one for
$g_n(x)$ and $h_n(x)$, where
\begin{equation}\label{hn}
h_n(x):=\frac{\widetilde{g}_n(x)-(\nu(x),\mu(x))g_n(x)}{\tau_n(x)}=\sum_{k=0}^n\lambda_k(x)\eta_k,
\end{equation}
with
$$\mu(x)=(\mu_0(x),\dots,\mu_n(x)), \ \ \ \nu(x)=(\nu_0(x),\dots, \nu_n(x)),$$
$$ (\nu(x),\mu(x))=\E[g_n(x)\widetilde{g}_n(x)]= \sum_{k=0}^n\nu_k(x)\mu_k(x),$$
and
$$\lambda(x)=(\lambda_0(x),\dots,\lambda_n(x))=\frac{\nu(x)-(\nu(x),\mu(x))\mu(x)}{\tau_n(x)},$$
\begin{equation}\label{tau}
\tau_n(x)=||\nu(x)-(\nu(x),\mu(x))\mu(x)||=\left(\sum_{k=0}^n[\nu_k(x)-(\nu(x),\mu(x))\mu_k(x)]^2\right)^{1/2}.
\end{equation}
From the definitions of $\mu(x)$, $\nu(x)$, and $\lambda(x)$, it follows that
\begin{align}\label{lamsq}
\sum_{k=0}^n\lambda_k(x)^2&=\sum_{k=0}^n\left(\frac{\nu_k(x)-(\nu(x),\mu(x))\mu_k(x)}{\tau_n(x)}\right)^2
=1,
\end{align}
and
\begin{align}\label{lammu}
\sum_{k=0}^n\lambda_k(x)\mu_k(x)&=\sum_{k=0}^n\left(\frac{\nu_k(x)-(\nu(x),\mu(x))\mu_k(x)}{\tau_n(x)}  \right)\mu_k(x)
=0.
\end{align}
By the above calculations, along with the first equation in \eqref{munu21}, we achieve that the vectors $\mu(x)$ and $\lambda(x)$ satisfy equation \eqref{proplammu} of Theorem
\ref{denlammu}.

Let us also note
\begin{align}\nonumber
\tau_n(x)
&=\left(1-2(\nu(x),\mu(x))^2+ (\nu(x),\mu(x))^2 \right)^{1/2}\\
\nonumber
&=\sqrt{ 1-\frac{K_n^{(0,1)}(x,x)^2}{K_n(x,x)K_n^{(1,1)}(x,x)} }\\
\label{tau2}
&=\sqrt{\frac{K_n(x,x)K_n^{(1,1)}(x,x)- K_n^{(0,1)}(x,x)^2}{K_n(x,x)K_n^{(1,1)}(x,x) }}
\end{align}
and
\begin{align}
\nonumber
\lambda_k(x)&=\frac{\nu_k(x)-(\nu(x),\mu(x))\mu_k(x)}{\tau_n(x)}\\
\nonumber
&=\frac{\frac{f_k^{\
\prime}(x)}{\left(K_n^{(1,1)}(x,x)\right)^{1/2}}-\frac{K_n^{(0,1)}(x,x)}{\left(K_n(x,x)K_n^{(1,1)}(x,x)\right)^{1/2}}\frac{f_k(x)}{\left(K_n(x,x)\right)^{1/2}}}
{\sqrt{\frac{K_n(x,x)K_n^{(1,1)}(x,x)- K_n^{(0,1)}(x,x)^2}{K_n(x,x)K_n^{(1,1)}(x,x) }}}\\
\label{lambda2}
&=\frac{K_n(x,x)f_k^{\ \prime}(x)-K_n^{(0,1)}(x,x)f_k(x)}{\left[K_n(x,x)\left(K_n^{(1,1)}(x,x)K_n(x,x)-K_n^{(0,1)}(x,x)^2  \right)  \right]^{1/2}}.
\end{align}
Let $\widehat{D}_n(\xi,\eta;x)$ be the joint distribution density of $g_n(x)$ and $h_n(x)$.  By another change of variables we have
$$\widetilde{D}_n(\xi,\eta;x)=\frac{1}{\tau_n(x)}\ \widehat{D}_n\left(\xi,\frac{\eta-(\nu(x),\mu(x))\xi}{\tau_n(x)};x\right)$$
so that now the transformation of the Kac-Rice formula \eqref{density2} reduces to
\begin{equation}
\rho_n(x)=\sqrt{\frac{K_n^{(1,1)}(x,x)}{K_n(x,x)}}\ \tau_n(x)\int_{\R}|\eta|\ \widehat{D}_n(0,\eta;x)\ d\eta.
\end{equation}
Furthermore, since
\begin{align*}
\sqrt{\frac{K_n^{(1,1)}(x,x)}{K_n(x,x)}}\ \tau_n(x)
=\sqrt{\frac{K_n^{(1,1)}(x,x)K_n(x,x)-K_n^{(0,1)}(x,x)^2}{K_n(x,x)^2}}
=\mathcal K_n(x),
\end{align*}
we have completed the result of the theorem.

\subsection{Proof of Corollary \ref{intGauss}}

Assume for all $s\in \R$ we have that the common characteristic function for the i.i.d.~random variables $\{\eta_k\}$ is
$\phi(s)=\exp(-ax^2)$,
where $a\in (0,\infty)$ is a fixed number.  Note that in this case, the variance of the i.i.d.~ random variables $\{\eta_k\}$ is $2a$.  This adjustment makes
\begin{equation*}
\E[P_n(x)^2]=2aK_n(x,x),\quad \E[P_n(x)P_n^{\prime}(x)]=2aK_n^{(0,1)}(x,x),
\end{equation*}
and
$$ \E[P_n^{\prime}(x)^2]=2aK_n^{(1,1)}(x,x).$$
Thus in the formula for $\mathcal K_n(x)$ defined at \eqref{Kn}, this extra factor for the variance cancels out algebraically.

For $\widehat{D}_n(\xi,\eta;x)$ the joint distribution function of $g_n(x)$ and $h_n(x)$ given by the first equation in \eqref{gn} and \eqref{hn} respectively, we have that the characteristic function is
\begin{equation*}
\Phi_n(\alpha,\beta)=\int_{\R}\int_{\R}\widehat{D}_n(\xi,\eta;x)e^{i\alpha \xi+i\beta \eta}\ d\xi d\eta.
\end{equation*}
Let $\omega_k=\alpha \mu_k(x)+\beta\lambda_k(x)$, $k\in \{0,1,\dots,n\}$.  By the properties \eqref{proplammu} in Theorem \ref{denlammu}, it follows that  $\sum_{j=0}^n\omega_k^2=\alpha^2+\beta^2:=|\gamma|^2$, with $\gamma=(\alpha,\beta)$.  Using the above assumption that characteristic function $\phi$ is a scaled mean zero Gaussian, and the assumption that the random variables are i.i.d.~we see that
\begin{equation*}
\Phi_n(\gamma)=\prod_{k=0}^n\phi(\omega_k)= \exp\left(-\sum_{k=0}^na\omega_k^2\right)= \exp(-a|\gamma|^2).
\end{equation*}

From the result of Theorem \ref{denlammu}, to complete the desired equality it suffices to compute
\begin{equation*}
\int_{\R}|\eta|\widehat{D}_n(0,\eta;x)\ d\eta,
\quad
\text{where}\quad
\widehat{D}_n(0,\eta;x)=\frac{1}{(2\pi)^2}\int_{\R^2}\exp(-i\beta \eta)\Phi_n(\gamma)\ d\gamma.
\end{equation*}
To this end, observe that
\begin{align*}
\widehat{D}_n(0,\eta;x)
&=\frac{1}{4a\pi}\exp\left(\frac{-\eta^2}{4a}\right).
\end{align*}
Therefore
\begin{align*}
\int_{\R}|\eta|\widehat{D}_n(0,\eta;x)\ d\eta=\frac{1}{4a\pi}\int_{\R}|\eta|\exp\left(\frac{-\eta^2}{4a}\right)\ d\eta
&=\frac{1}{\pi},
\end{align*}
and hence giving the desired result.

\subsection{Proof of Theorem \ref{mainthm}}
As in the proof of Theorem \ref{intGauss}, since the random variables $\{\eta_k\}$ are independent and
$$g_n(x)=\sum_{k=0}^n\mu_k(x)\eta_k\ \ \ \text{and}\ \ \ h_n(x)=\sum_{k=0}^n\lambda_k(x)\eta_k,$$
the joint distribution function $\widehat{D}_n(\xi,\eta;x)$ of $g_n$ and $h_n$ has characteristic function satisfying
\begin{equation}\label{Phiprod}
\Phi_n(\gamma)=\Phi_n(\alpha,\beta)=\int_{\R}\int_{\R}\widehat{D}_n(\xi,\eta;x)e^{i\alpha \xi+i\beta \eta}\ d\xi d\eta=\prod_{k=0}^n\phi(\omega_k),
\end{equation}
where $\gamma=(\alpha,\beta)$ and
$\omega_k=\mu_k(x)\alpha+\lambda_k(x)\beta$,  $k=0,\dots,n.$

We now present our version of Lemma 4.1 from \cite{BD}.  The main differences in the presented proof of the lemma is that we do not require $\omega_k^2=\mathcal O(n^{-1/2}|\gamma|^2)$, and we give a slightly different partition of the index set $\{k:k=0,1,\dots, n\}$ for the product \eqref{Phiprod}.
\begin{lemma}\label{lem4.1}
If $\phi(s)$ satisfies \eqref{hyphi1}, then
\begin{equation}
|\Phi_n(\gamma)|\leq \frac{1}{(1+a_0|\gamma|^2)^L},
\end{equation}
with $a_0=a/T$ and $L=Tq$, where  $2\leq T \leq n+1$, is the number of partitions of the index set $\{k:k=0,1,\dots, n\}$ constructed in the proof.
\end{lemma}
\begin{proof}
Using the assumption \eqref{hyphi1} on the common characteristic function $\phi(s)$ of $\{\eta_k\}$ and \eqref{Phiprod} we have
\begin{equation}\label{Phiprod2}
|\Phi_n(\gamma)|\leq \prod_{k=0}^n\frac{1}{(1+a\omega_k^2)^q}.
\end{equation}

Since the relations \eqref{proplammu} of Theorem \ref{denlammu} and the definition of $\omega_k$ give
$\sum_{k=0}^n\omega_k^2=\alpha^2+\beta^2=|\gamma|^2$, we partition the index set $\{k:k=0,1\dots, n\}$ into $T$ groups, with  $2\leq T \leq n+1$, which we call $M_j$ to give on each group the following
\begin{equation*}
\sum_{k\in M_j}\omega_k^2\geq \frac{1}{T}|\gamma|^2.
\end{equation*}
The partition and estimate yield
\begin{equation*}
1+\frac{a}{T}|\gamma|^2\leq 1+a\sum_{k\in M_j}\omega_k^2\leq \prod_{k\in M_j}(1+a\omega_k^2).
\end{equation*}
Thus using the above bound and \eqref{Phiprod2},  we have
\begin{align*}
|\Phi_n(\gamma)|&\leq \prod_{j=1}^T\prod_{k\in M_j}\frac{1}{(1+a\omega_k^2)^q}\\
&\leq
\prod_{j=1}^T\frac{1}{(1+\frac{a}{T}|\gamma|^2)^q}\\
&=\frac{1}{(1+\frac{a}{T}|\gamma|^2)^{Tq}}\\
&=\frac{1}{(1+a_0|\gamma|^2)^{L}},
\end{align*}
where $a_0=a/T$ and taking $L=qT$.
\end{proof}

Our next needed lemmas are modifications of Lemma 4.2 and Lemma 4.3 of \cite{BD}.  The modifications allow us to keep track of all constants.  The first lemma is already given at equation (4.15) of \cite{BD}.
\begin{lemma}[Bleher and Di, Lemma 4.2 \cite{BD}]\label{lem4.21}
If $\phi(s)$ satisfies \eqref{hyphi1} and \eqref{hyphi2}, then
\begin{equation}\label{Phi1}
\left| \frac{\partial\Phi_n(\gamma)}{\partial \beta} \right|\leq \frac{C_2|\gamma|}{(1+a_0|\gamma|^2)^L}.
\end{equation}
\end{lemma}

\begin{lemma}\label{lem4.22}
If $\phi(s)$ satisfies \eqref{hyphi1} and \eqref{hyphi2}, then
\begin{equation}\label{Phi3}
\left| \frac{\partial^3\Phi_n(\gamma)}{\partial \beta^3} \right|\leq \frac{C_3}{(1+a_0|\gamma|^2)^L}+\frac{2C_2^2|\gamma|}{(1+a_0|\gamma|^2)^L}.
\end{equation}
\end{lemma}
\begin{proof}
Observe that
\begin{align}\nonumber
\frac{\partial^3}{\partial \beta^3}\Phi_n(\gamma)&=\frac{\partial^3}{\partial \beta^3}\prod_{k=0}^n\phi(\omega_k)\\
\label{Phi31}
&=\sum_{k=0}^n\lambda_k^3\phi^{\prime \prime \prime}(\omega_k)\prod_{l\neq k}\phi(\omega_l) \\
\label{Phi32}
&\ \ \ +\sum_{k=0}^n\sum_{i\neq k}\lambda_i^2\lambda_k\phi^{\prime \prime}(\omega_i)\phi^{\prime}(\omega_k)\prod_{l\neq i,k}\phi(\omega_l)\\
\label{Phi33}
&\ \ \ +\sum_{k=0}^n\sum_{i\neq k}\lambda_i\lambda_k^2\phi^{ \prime}(\omega_i)\phi^{\prime \prime}(\omega_k)\prod_{l\neq i,k}\phi(\omega_l).
\end{align}
As in proof of Lemma \ref{lem4.1} it follows that
\begin{equation*}
\left| \prod_{l\neq k}\phi(\omega_l) \right|,\left| \prod_{l\neq i,k}\phi(\omega_l) \right|\leq \frac{1}{(1+a_0|\gamma|^2)^L}.
\end{equation*}
Using the above and the estimate on $\phi^{\prime \prime \prime}$ in \eqref{hyphi2} we achieve
\begin{align}\nonumber
\left|\sum_{k=0}^n\lambda_k^3\phi^{\prime \prime \prime}(\omega_k)\prod_{l\neq k}\phi(\omega_l)\right|
&\leq \frac{C_3}{(1+a_0|\gamma|^2)^L}\sum_{k=0}^n|\lambda_k^3| \\
\nonumber
&\leq \frac{C_3}{(1+a_0|\gamma|^2)^L}\left(\sum_{k=0}^n\lambda_k^2\right)^2\\
\label{Phi311}
&= \frac{C_3}{(1+a_0|\gamma|^2)^L},
\end{align}
where we have appealed to Cauchy Schwarz in the second inequality and fact that the $\lambda_k$'s are normalized in the last equation.

To complete the estimates, let us first note that since $\E[\eta_k]=0$ for $k=0,\dots,n$, the characteristic function $\phi(s)$ of $\eta_k$ satisfies $\phi^{\prime}(0)=0$.
Hence under the assumption of \eqref{hyphi2}, this gives that for $s\in \R$ we have
$| d\phi(s)/ds|\leq C_2|s|.$  With this in mind, estimating as previously done gives
\begin{align}\nonumber
\left| \sum_{k=0}^n\sum_{i\neq k}\lambda_i^2\lambda_k\phi^{\prime \prime}(\omega_i)\phi^{\prime}(\omega_k)\prod_{l\neq i,k}\phi(\omega_l)  \right|&\leq
\frac{C_2^2}{(1+a_0|\gamma|^2)^L}\sum_{k=0}\sum_{i\neq k}|\lambda_i^2\lambda_k\omega_k|\\
\nonumber
&\leq \frac{C_2^2}{(1+a_0|\gamma|^2)^L}\sum_{k=0}|\lambda_k\omega_k|\sum_{i=0}^n\lambda_i^2\\
\nonumber
& \leq \frac{C_2^2}{(1+a_0|\gamma|^2)^L}\left(\sum_{k=0}^n\lambda_k^2 \ \sum_{j=0}^n\omega_j^2\right)^{1/2}\\
\label{Phi321}
&=\frac{C_2^2|\gamma|}{(1+a_0|\gamma|^2)^L}.
\end{align}
The estimate for \eqref{Phi33} is done similarly and has the same bound as the above. Combining \eqref{Phi311} and twice that of \eqref{Phi321} gives the desired \eqref{Phi3}.
\end{proof}

\begin{lemma}\label{lem4.3}
If $\phi(s)$ satisfies \eqref{hyphi1} and \eqref{hyphi2}, then
\begin{equation}\label{D1}
\left| \widehat{D}_n(0,\eta;x) \right|\leq \frac{K_1}{(1+|\eta|)},\quad K_1=\frac{1}{2\pi aq}+\frac{C_2}{(2aq)^{3/2}},
\end{equation}
and
\begin{equation}\label{D3}
\left| \widehat{D}_n(0,\eta;x) \right|\leq \frac{K_2}{(1+|\eta|^3)}, \quad
K_2=\frac{1}{2\pi aq}+\frac{C_2^2}{\sqrt{2}(aq)^{3/2}}+\frac{C_3}{2\pi aq}.
\end{equation}
\end{lemma}
\begin{proof}
Since
$$\Phi_n(\gamma)=\int_{-\infty}^{\infty}\int_{-\infty}^{\infty}\widehat{D}_n(\xi,\eta;x)e^{i\alpha \xi + i\beta \eta}\ d\xi \ d\eta,$$
differentiating and using Fourier inversion gives
$$\eta^k\widehat{D}_n(0,\eta;x)=\frac{(-i)^k}{(2\pi)^2}\int_{\R^2}e^{-i\beta \eta}\frac{\partial^k \Phi_n(\gamma)}{\partial \beta^k}\ d\gamma.$$
Using the above with Lemma \ref{lem4.1} and Lemma \ref{lem4.21} yields
\begin{align}\nonumber
(1+|\eta|)|\widehat{D}_n(0,\eta;x)|&\leq \frac{1}{(2\pi)^2} \int_{\R}|\Phi_n(\gamma)|\ d\gamma+\frac{1}{(2\pi)^2}\int_{\R^2}\left|\frac{\partial \Phi_n(\gamma)}{\partial \beta} \right|\ d\gamma\\
\nonumber
&\leq \frac{1}{(2\pi)^2} \int_{\R^2}\frac{1}{(1+a_0|\gamma|^2)^L}\ d\gamma \\
\nonumber
& \ \ +\frac{C_2}{(2\pi)^2}\int_{\R^2}\frac{|\gamma|}{(1+a_0|\gamma|^2)^L}\ d\gamma\\
\label{Est1}
&=\frac{1}{4\pi a_0(L-1)}+\frac{C_2\Gamma(L-3/2)}{8\pi^{1/2}a_0^{3/2}\Gamma(L)},
\end{align}
where $\Gamma$ is the usual Gamma Function.

Similarly with aid of Lemma \ref{lem4.22}, we see
\begin{align}\nonumber
(1+|\eta|^3)|\widehat{D}_n(0,\eta;x)|&\leq \frac{1}{(2\pi)^2}\int_{\R}|\Phi_n(\gamma)|\ d\gamma+\frac{1}{(2\pi)^2}\int_{\R^2}\left|\frac{\partial^3 \Phi_n(\gamma)}{\partial \beta^3} \right|\
d\gamma\\
\nonumber
&\leq \frac{1}{(2\pi)^2}\int_{\R^2}\frac{1}{(1+a_0|\gamma|^2)^L}\ d\gamma \\
\nonumber
& \ \ +\frac{2C_2^2}{(2\pi)^2}\int_{\R^2}\frac{|\gamma|}{(1+a_0|\gamma|^2)^L}\ d\gamma\\
\nonumber
& \ \ +
\frac{C_3}{(2\pi)^2}\int_{\R^2}\frac{1}{(1+a_0|\gamma|^2)^L}\ d\gamma\\
\label{Est2}
&=\frac{1}{4\pi a_0(L-1)}+\frac{C_2^2\Gamma(L-3/2)}{4\pi^{1/2}a_0^{3/2}\Gamma(L)}+\frac{C_3}{4a_0\pi(L-1)}.
\end{align}

To complete the estimates, note that
\begin{align}\label{est1}
\frac{1}{4\pi a_0(L-1)}=\frac{ T}{4\pi a(Tq-1)}=\frac{1}{4\pi aq}\frac{Tq}{Tq-1}\leq \frac{1}{2\pi aq},
\end{align}
since the function $x/(x-1)$ is decreasing and $Tq\geq 2$ given that $T\geq 2$ along with the assumption that $q\geq 1$.  This estimate also gives
\begin{equation}\label{est2}
\frac{C_3}{4a_0\pi(L-1)}\leq \frac{C_3}{2aq \pi}.
\end{equation}
Also observe
\begin{equation}\label{est3}
\frac{\Gamma(L-3/2)}{\pi^{1/2}a_0^{3/2}\Gamma(L)}=\frac{T^{3/2}\Gamma(Tq-3/2)}{\pi^{1/2}a^{3/2}\Gamma(Tq)}=\frac{(Tq)^{3/2}\Gamma(Tq-3/2)}{\pi^{1/2}(aq)^{3/2}\Gamma(Tq)}
\leq \frac{2^{3/2}}{(aq)^{3/2}},
\end{equation}
by the function $x^{3/2}\Gamma(x-3/2)/\Gamma(x)$ being decreasing on $[2,\infty)$, and at $x=2$ evaluating to $2\sqrt{2\pi}$.

Combining \eqref{est1}, \eqref{est2}, and \eqref{est3} with \eqref{Est1} and \eqref{Est2}, completes the estimates need for \eqref{D1} and \eqref{D3}.
\end{proof}

We are now finally ready to give the proof of the main theorem
\begin{proof}[Proof of Theorem \ref{mainthm}]
From Theorem \ref{denlammu} we have
\begin{align*}
\rho_n(x)&=\mathcal K_n(x)\int_{\R}|\eta|\ \widehat{D}_n(0,\eta;x)\ d\eta.
\end{align*}

Estimating the integral on the right-hand side and using Lemma \ref{lem4.3} yields
\begin{align*}
\int_{\R}|\eta|\ \widehat{D}_n(0,\eta;x)\ d\eta
&= \left(\int_{-1}^1+\int_{|\eta|\geq 1}\right)|\eta|\ \widehat{D}_n(0,\eta;x)\ d\eta\\
&\leq K_1 \int_{-1}^1\frac{|\eta|}{1+|\eta|}\ d\eta +K_2\int_{|\eta|\geq 1}\frac{|\eta|}{1+|\eta|^3}\ d\eta\\
&=2K_1 (1-\log 2)+\frac{2}{9}K_2(\sqrt{3}\pi +\log 8)\\
&=\frac{1}{aq}(k_1+C_3k_2)+\frac{1}{(aq)^{3/2}}[C_2(k_3+C_2k_4)],
\end{align*}
where
\begin{align*}
k_1&=\frac{1}{\pi}\left(1-\log 2+\frac{\pi\sqrt{3}+\log 8}{9}\right)= 0.36367\dots,\\
k_2&=\frac{1}{9\pi}(\pi\sqrt{3}+\log 8)= 0.265995\dots,\\
k_3&=\frac{1}{\sqrt{2}}(1-\log 2)= 0.216978\dots,\\
k_4&=\frac{\sqrt{2}}{9}(\pi \sqrt{3}+\log 8)= 1.18179\dots,
\end{align*}
and thus completes the proof.
\end{proof}

\subsection{Proof of Theorem \ref{Berg1}}

For Bergman polynomials $p_k(z)=\sqrt{(k+1)/\pi}z^k$ observe that
\begin{align}
\nonumber
K_n(z,w)&=\sum_{k=0}^np_k(z)\overline{p_k(w)}\\
\nonumber
&=\frac{1}{\pi}\sum_{k=0}^n(k+1)(z\overline{w})^k \\
\label{bp}
&=\frac{1-(z\overline{w})^{n+1}(2-z\overline{w})}{\pi (1-z\overline{w})^2}-\frac{n(z\overline{w})^{n+1}}{\pi (1-z\overline{w})}.
\end{align}
Thus for $x\in \R\setminus \{\pm 1\}$, as $n\rightarrow \infty$ locally uniformly on the respective domains we have
\begin{equation}\label{bx}
 K_n(x,x)=\begin{cases}
      \displaystyle\frac{1}{\pi(1-x^2)^2}+o(1), & |x|<1, \\[2ex]
      \displaystyle\frac{nx^{2n+2}}{\pi (x^2-1)}(1+o(1)), & |x|>1.
   \end{cases}
\end{equation}
Taking derivatives of \eqref{bp} and then evaluating on the diagonal to form the other needed kernels $K_n^{(0,1)}(x,x)$ and $K_n^{(1,1)}(x,x)$, after algebraic simplification one sees
\begin{equation}\label{nbx}
K_n(x,x)K_n^{(1,1)}(x,x)-(K_n^{(0,1)}(x,x))^2=\begin{cases}
      \displaystyle\frac{2}{\pi^2(1-x^2)^6}+o(1), & |x|<1, \\[2ex]
      \displaystyle\frac{n^2x^{4n+4}}{\pi^2 (x^2-1)^4}(1+o(1)), & |x|>1.
   \end{cases}
\end{equation}
Therefore, combining \eqref{bx} and \eqref{nbx} locally uniformly for $x\in \R\setminus \{\pm 1\}$ we have
\begin{align*}
\lim_{n\rightarrow \infty}\mathcal K_n(x)&=\lim_{n\rightarrow \infty}\sqrt{\frac{K_n(x,x)K_n^{(1,1)}(x,x)-K_n^{(0,1)}(x,x)^2}{K_n(x,x)^2}}\\
&=\begin{cases}
      \displaystyle\frac{\sqrt{2}}{1-x^2} & |x|<1, \\[1.5ex]
      \displaystyle\frac{1}{x^2-1} & |x|>1.
   \end{cases}
\end{align*}
When $x=\pm 1$, using summation formulas the kernels can be evaluated directly to give
\begin{equation*}
K_n(x,x)=\frac{(n+2)(n+1)}{2},\ \ \  K_n^{(0,1)}(x,x)^2=\left(\frac{n(n+1)(n+2)}{3}\right)^2,
\end{equation*}
and 
$$ K_n^{(1,1)}(x,x)=\frac{n(n+1)(n+2)(3n+1)}{12}.$$
Hence
\begin{equation*}
\mathcal K_n(x)=\frac{1}{3}\sqrt{\frac{n(n+3)}{2}}, \quad \text{for}\quad x=\pm 1.
\end{equation*}

\subsection*{Acknowledgements}
The author would like to thank his Ph.D. advisor Igor Pritsker for his helpful conversations concerning this project, and financial support from the Vaughn Foundation on behalf of Anthony Kable.


\begin{thebibliography}{HD82}




\normalsize
\baselineskip=17pt


\bibitem{HY}
H. Aljubran and M. Yattselev, \emph{An asymptotic expansion for the expected number of real zeros of real random polynomials spanned by OPUC}, J. Math. Anal. Appl. 469 (2019) 428--446.


\bibitem{BY}
T. Bayraktar, \emph{Equidistribution of zeros of random holomorphic sections}, Indiana Univ. Math. J. 65 (2016), 1759--1793.


\bibitem{BD}
P. Bleher and X. Di, \emph{Correlations between zeros of non-Gaussian random polynomials}, Intl. Math. Res. Not., 2004 No. 46


\bibitem{BP}
A. Bloch and G. P\'{o}lya, \emph{On the roots of a certain algebraic equation}, Proc. Lond. Math. Soc. 33 (1932) 102--114.



\bibitem{BR}A. T. Bharucha-Reid and M. Sambandham, \emph{Random polynomials}, Academic Press, Orlando, 1986.

\bibitem{D}
M. Das, \emph{Real zeros of a random sum of orthogonal polynomials}, Proc. Amer. Math. Soc. 1 (1971) 147--153.

\bibitem{DB}
M. Das and S. Bhatt, \emph{Real roots of random harmonic equations}, Indian J. Pure Appl. Math. 13 (1982) 411--420.



\bibitem{DNV}
Y. Do, H. Nguyen, V. Vu, \emph{Real roots of random polynomials: expectation and repulsion}, Proc. Lond. Math. Soc. Vol. 111, 6 (2015), 1231--1260.

\bibitem{DNV2}
Y. Do, H. Nguyen, V. Vu, \emph{Roots of random polynomials with arbitary coefficients},  arXiv:1507.04994v3 (2016).

\bibitem{EK}
A. Edelman and E. Kostlan, \emph{How many zeros of a random polynomial are real?}, Bull. Amer. Math. Soc. 32 (1995) 1-37.

\bibitem{EO}
P. Erd\H{o}s and A. Offord, \emph{On the number of real roots of a random algebraic equation}, Proc. Lond. Math. Soc. 6 (1956) 139--160.

\bibitem{ET}P. Erd\H{o}s and P. Tur\'an, \emph{On the distribution of roots of polynomials}, Ann. Math.  51 (1950), 105--119.


\bibitem{Fa}K. Farahmand, \emph{Topics in random polynomials}, Pitman Res. Notes Math.  393 (1998).








\bibitem{HM}
J. Hammersley, \emph{The zeros of a random polynomial}, Proc. of the Third Berk. Sym. on Math. Stat. and Prob. 1954-1955 vol. II, University of Cal. Press, Berkeley and Los Angeles (1956) 89--111.




\bibitem{IM1}
I. Ibragimov and N. Maslova, \emph{The average number of zeros of random polynomials}, Vestnik Leningrad. Univ. 23 (1968), 171--172.

\bibitem{IM2}
I. Ibragimov and N. Maslova, \emph{The mean number of real zeros of random polynomials. I. Coefficients with zero mean}, Theor. Probability Appl.  16 (1971), 228--248.

\bibitem{IZ}I. Ibragimov and O. Zeitouni, \emph{On roots of random polynomials}, Trans. Amer. Math. Soc.  349 (1997), 2427--2441.

\bibitem{IZa}I. Ibragimov and D. Zaporozhets, \emph{On distribution of zeros of random polynomials in complex plane}, Prokhorov and contemporary probability theory, Springer Proc. Math. Stat.  33 (2013), 303--323.


\bibitem{K1}
M.  Kac,  \emph{On  the  average  number  of  real  roots  of  a  random  algebraic  equation},  Bull.
Amer. Math. Soc. 49 (1943) 314--320.

\bibitem{K2}
M. Kac, \emph{On the average number of real roots of a random algebraic equation II}, Proc.
Lond. Math. Soc. 50 (1948) 390--408.



\bibitem{LO1}
J. Littlewood and A. Offord, \emph{On the distribution of the zeros and a-values of a random integral function I}, J. Lond. Math. Soc. 20 (1945), 120--136.

\bibitem{LO2}
J. Littlewood and A. Offord, \emph{On the distribution of the zeros and values of a random integral function II}, Ann. Math. 49 (1948), 885--952. Errata, 50 (1949), 990--991, 976, 35--58.

\bibitem{LO3}
J. Littlewood and A. Offord, \emph{On the number of real roots of a random algebraic equation}, J. Lond. Math. Soc. 13 (1938) 288--295.


\bibitem{LO4}
J. Littlewood and A. Offord, \emph{On the number of real roots of a random algebraic equation II}, Proc. Cambridge Philos. Soc. 35 (1939), 133--148.

\bibitem{LO5}
J. Littlewood and A. Offord, \emph{On the number of real roots of a random algebraic equation III}, Rec. Math. [Mat. Sbornik] N.S. 54 (1943), 277--286.


\bibitem{LS1}
B. Logan and L. Shepp, \emph{Real zeros of random polynomials}, Proc. London Math. Soc., 13:29--35, 1945.

\bibitem{LS2}
B. Logan and L. Shepp, \emph{Real zeros of random polynomials II}, Proc. London Math. Soc., 18:308--314, 1968.


\bibitem{LPX}
D. Lubinsky, I. Pritsker, and X. Xie, \emph{Expected number of real zeros for random linear combinations of orthogonal polynomials}, Proc.  Amer. Math. Soc. 144 (2016) 1631--1642.

\bibitem{LPX2}
D. Lubinsky, I. Pritsker, and X. Xie, \emph{Expected number of real zeros for random orthogonal polynomials}, Math. Proc. Cambr. Philos. Soc. 164 (2018), 47--66.


\bibitem{NNV}
H. Nguyen, O. Nguyen, and V. Vu, \emph{On the number of real roots of random polynomials}, Commun. Contemp. Math. 18, 1550052 (2016).

\bibitem{PRIT}
I. Pritsker, \emph{Expected zeros of random orthogonal polynomials on the real line}, Jaen J. Approx. 9 (2017), 1--24.

\bibitem{PritYgr15}
I. Pritsker and A. Yeager, \emph{Zeros of polynomials with random coefficients}, J. Approx. Theory 189 (2015), 88--100.



\bibitem{R}
S. Rice, \emph{Mathematical theory of random noise}, Bell System Tech J. 25 (1945) 46--156.




\bibitem{Stevens}
D. Stevens, \emph{The average number of real zeros of a random polynomial}, Comm. Pure Appl. Math. 22 (1969), 457--477.


\bibitem{TV}
T. Tao and V. Vu, \emph{Local universality of zeros of random polynomials}, Int. Math. Res. Not. (2014) 1--84.


\bibitem{CZRS}
R.  Vanderbei, \emph{The complex zeros of random sums}, arXiv: 1508.05162v1   Aug. 21, 2015.


\bibitem{WL}
J. Wilkins Jr., \emph{An asymptotic expansion for the expected number of real zeros of a random polynomial}, Proc. Amer. Math. Soc. 103 (1988) 1249--1258.


\bibitem{YY}
M. Yattselev and A. Yeager, \emph{Zeros of real random polynomials spanned by OPUC}, to appear in Indiana University Mathematics Journal.

\bibitem{AY2}
A. Yeager, \emph{Zeros of random orthogonal polynomials with complex Gaussian coefficients}, Rocky Mount. J. Math. (2018) 48 no. 7, 2385--2403.


\end{thebibliography}
\end{document}